\documentclass{amsart}

\makeatletter
\def\BState{\State\hskip-\ALG@thistlm}
\makeatother

\usepackage{color}
\usepackage[normalem]{ulem}
\usepackage{enumitem}

\usepackage{amsthm}
\usepackage{amsmath}
\usepackage{mathtools}
\usepackage{amssymb}
\usepackage{commath}
\usepackage{bigints}
\usepackage{systeme}

\usepackage{soul}

\newtheorem{thm}{Theorem}

\newtheorem{remark}{Remark}

\newtheorem{lemma}{Lemma}[section]

\newcommand{\I}{\mathcal{D}}

\newcommand{\Th}{\mathcal{T}_h} 
\newcommand{\Vh}{\mathbb{V}_h} 
\newcommand{\Vn}{\mathbb{V}_h^n} 
\newcommand{\Wn}{\mathbb{W}_h^n} 
\newcommand{\Hr}{\mathbb{H}^r(\Omega)}
\newcommand{\Hrr}{\mathbb{H}^1(\Omega)}
\newcommand{\Hrp}{\mathbb{H}^{r+1}(\Omega)}
\newcommand{\Hrse}{\mathbb{H}^{r-2s+\varepsilon}(\Omega)}
\newcommand{\Ha}{\mathbb{H}^{r+2\alpha^*-2s}(\Omega)}
\newcommand{\Hae}{\mathbb{H}^{r+2\alpha^*-2s+2\varepsilon}(\Omega)}

\DeclareMathOperator{\Span}{span}
\DeclareMathOperator{\argmax}{argmax}

\title{Reduced basis approximations of the solutions to spectral fractional diffusion problems}
\author{Andrea Bonito, Diane Guignard, Ashley R. Zhang}
\address{\{AB, DG, AZ\}, Dept. of Mathematics, Texas A$\&$M University,  College Station, TX-77843, \{bonito,dguignard,ashleyrzhang\}@math.tamu.edu}

\date{\today}


\thanks{%
 This research was supported by the NSF grant DMS-1817691 (AB-DG-AZ);  DG was supported by the Swiss National Science Foundation grant P2ELP2-175056 and IAMCS at TAMU}
\begin{document}
	
\maketitle

\begin{abstract}
We consider the numerical approximation of the spectral fractional diffusion problem based on the so called Balakrishnan representation. The latter consists of an improper integral approximated via quadratures. At each quadrature point, a reaction-diffusion problem must be approximated and is the method bottle neck. In this work, we propose to reduce the computational cost using a reduced basis strategy allowing for a fast evaluation of the reaction-diffusion problems. The reduced basis does not depend on the fractional power $s$ for $0<s_{\min}\leq s \leq s_{\max}<1$. It is built \emph{offline} once for all and used \emph{online} irrespectively of the fractional power. We analyze the reduced basis strategy and show its exponential convergence. The analytical results are illustrated with insightful numerical experiments.
\end{abstract}
	
\section{Introduction}
	
Nonlocal models have recently received a great attention due to their apparent ability to capture novel effects such as in mechanics \cite{teodor2014fractional} and in particular in peridynamics \cite{silling2000,ha2010studies}, turbulence \cite{chen2006speculative}, biophysics \cite{bueno2014fractional} and image denoising \cite{gatto2015numerical}, to mention a few.

In most of the applications, the type of nonlocal interactions are different and their scaling laws are unknown. Initiated by the work in \cite{AO18}, an algorithm is proposed and analyzed in \cite{antil2016optimization} to  identify the fractional power $s \in [s_{\min},s_{\max}]$ governing the state equation in an optimization framework. As expected, the algorithm exploits the smoothness of the map $s \mapsto (-\Delta)^{-s} f$ and requires many (costly) evaluations of  $(-\Delta)^{-s} f$ for different $s \in (0,1)$. Several numerical methods are available to approximate $(-\Delta)^{-s} f$ and we refer to the surveys \cite{BBNOS18,lischke2018fractional} for  the description of different fractional Laplacians along with their numerical approximations. Here, for $f \in L^2(\Omega)$ and $\Omega$ a Lipschitz domain of $\mathbb R^d$, $d=1,2,3$, we set
\begin{equation} \label{def:prob}
u(s):=(-\Delta)^{-s} f:=\sum_{k=1}^{\infty}\lambda_k^{-s} f_k\psi_k,
\end{equation}
where $\{\lambda_k,\psi_k\}_{k\in\mathbb{N}} \subset  \mathbb{R}^+ \times  H^1_0(\Omega)$ are the eigenpairs of $(-\Delta)$ and $f_k := \int_\Omega f \psi_k$. The eigenfunctions $\{\psi_k\}_{k\in \mathbb N}$ are chosen orthogonal in $H_0^1(\Omega)$ and orthonormal in $L^2(\Omega)$. In \eqref{def:prob} the fractional operator is referred to as the spectral fractional Laplacian and is the one considered in this work. It is worth mentioning that the methodology proposed here is not limited to the Laplacian operator and can be easily extended to regularly accretive operators as in \cite{BP17}.
	
In this work, we follow the approach proposed in \cite{BP15} to approximate \eqref{def:prob}, see also \cite{BP17}, which is based on the Dunford-Taylor-Balakrishnan representation 
\begin{equation*}
u(s)=\frac{\sin(s\pi)}{\pi}\int_{-\infty}^{\infty}e^{(1-s)y} w(y)dy,
\end{equation*}
where $w(y) \in H^1_0(\Omega)$ solves
\begin{equation}\label{eq:wy}
(e^y I -\Delta)w(y) = f.
\end{equation}
Originally introduced in \cite{BP15} and later improved in \cite{BLP18}, a sinc quadrature coupled with a standard finite element method is used for the approximation of the integration in the variable $y$. For $k>0$, it reads
\begin{equation} \label{def:uk}
u(s) \approx u_{h,k}(s)=\frac{k\sin(s\pi)}{\pi}\sum_{l=-M_s}^{N_s}e^{(1-s)y_l}w_h(y_l)
\end{equation}
with $y_l:=lk$,
\begin{equation} \label{def:NsMs}
M_s:=\left\lceil \frac{\pi^2}{(1-s)k^2} \right\rceil, \quad N_s:=\left\lceil \frac{\pi^2}{sk^2} \right\rceil
\end{equation}
and where $w_h(y_l)\in  \mathbb V_h$ are standard finite element approximations of $w(y_l)$.
	
The numerical approximation of $(-\Delta)^{-s} f$ requires $M_s+N_s+1$ finite element solves to determine $w_h(y_l)$, $y_l\in[-M_sk,N_sk]$. This can become prohibitive when the computation of $(-\Delta)^{-s} f$ is needed for many values of $s$ such as within an optimization loop as mentioned above. The reduced basis method seems to be a natural approach to reduce the computational cost when approximating the parametrized reaction-diffusion problems \eqref{eq:wy}. In fact, reduced basis method for this type of one dimensional parametric elliptic partial differential equation has already been partially analyzed in \cite{maday2002priori,MPT02} and recently in \cite{DS19} from which part of our analysis is inspired.
	
A (\emph{weak}) greedy strategy is advocated (\emph{offline stage}) to iteratively select snapshots $w_h(y^l)$, \mbox{$l=1,...,n <\!\!< \dim(\mathbb V_h)$}, defining the $s$-independent reduced basis space $\Vn := \Span\{w_h(y^1),\ldots,w_h(y^n)\}$. Galerkin approximations $w_h^n(y_l) \in \Vn$ of $w_h(y_l)$ can then be easily computed (\emph{online stage}) to produce a reduced basis approximation of $u_{h,k}(s)$ 
\begin{equation}\label{def:ukn}
u_{h,k}^n(s) := \frac{k\sin(s\pi)}{\pi}\sum_{l=-M_s}^{N_s}e^{(1-s)y_l}w_h^n(y_l).
\end{equation}
We point out that one of the difficulty faced in this study is that the approximation of the parametric elliptic partial differential equation~\eqref{eq:wy} is required for $y$ in the parametric domain $[-M_sk,N_sk]$ whose length increases as the sinc quadrature parameter $k$ decreases (improving the precision of the algorithm). 

The proposed algorithm provides an approximation of the entire map $s \mapsto u(s)$,  $s\in[s_{\min},s_{\max}]$ using the same reduced basis space $\Vn$. Our main result is Theorem~\ref{cor:err_D} which guarantees an exponential convergence of the reduced basis approximation $u_{h,k}^n(s)$ toward $u_{h,k}(s)$ in a wide range of Sobolev norms, uniformly in the fractional power $s\in[s_{\min},s_{\max}]$.

We end this introduction by noting that the idea of using the reduced basis method for fractional problems has been recently proposed for instance in \cite{DS19} and \cite{DACCN19}. In \cite{DS19}, the reduced basis is used for the approximation of interpolation norms, as well as evaluations of both types $s \mapsto (-\Delta)^s u$ with $u$ fixed and variable $s\in(0,1)$ and $u \mapsto (-\Delta)^s u$. A reduced basis space based on best rational approximations for a reaction diffusion problem similar to the one satisfied by $w_{h}$ is proposed and exponential convergence of the approximation with respect to the dimension of the reduced basis space is obtained. Worth mentioning, the numerical method is based on the extension method \cite{NOS15} but seemingly apply to other approximation techniques. Actually, this method boils down to the approximation of several reaction diffusion problems as in \cite{BP15}. We take advantage of the technology developed in \cite{DS19} to derive an exponential decay in the approximation of \eqref{def:uk} by \eqref{def:ukn}. In \cite{DACCN19}, a similar approximation $u_{h,k}(s)$ is proposed for a different quadrature. Exponential decay of the reduced basis error is observed numerically but without analysis. In some sense, this work provides a mathematical justifications of the experimental observations in \cite{DACCN19}. Finally, we mention that the reduced basis method has also been used in \cite{BG19} to approximate the parametric PDEs $(-\Delta)^s u=f$, where $s$ is the parameter, in the case of the integral fractional Laplacian.

The rest of the paper is organized as follows. In Section~\ref{sec:frac}, we describe the numerical approximation of $u(s)$ by $u_{h, k}(s)$. Section 3 describes the construction of the reduced basis space and its corresponding error analysis - the main result of this work. Section 4 provides numerical experiments to illustrate the performance of the proposed methodology.

\section{Spectral Fractional Laplacian and its Numerical Approximations}\label{sec:frac}
	
We start with some notations. Let $\Hr$ be the interpolation space defined by
\begin{equation} \label{def:Hr}
\Hr:=\left\{\begin{array}{ll}
\left(L^2(\Omega),H_0^1(\Omega)\right)_r & \mbox{for } r\in [0,1] \\
H_0^1(\Omega)\cap H^r(\Omega) & \mbox{for } r\in (1,2],
\end{array}
\right.
\end{equation}
where $(\cdot,\cdot)_r$ denotes interpolation using the real method.

Notice that for the particular case $r=1$, we have
\begin{equation*}
\|v\|_{H_0^1(\Omega)}:=\|v\|_{\Hrr}=\|\nabla v\|_{L^2(\Omega)} \quad \forall v\in H_0^1(\Omega),
\end{equation*}
which is equivalent to the $H^1(\Omega)$ norm thanks to the Poincar\'e inequality
\begin{equation} \label{def:Poincare}
\|v\|_{L^2(\Omega)}\leq C_P\|\nabla v\|_{L^2(\Omega)} \quad \forall v\in H_0^1(\Omega).
\end{equation}
To simplify the notation, we write when $r=0$, $\|\cdot\|:=\|\cdot\|_{L^2(\Omega)}=\|.\|_{\mathbb H^0}$. Moreover, $a\lesssim b$ means that $a\leq Cb$ for a constant $C$ that does not depend on $a$, $b$ and the discretization parameters and whose value might change at each occurrence. Also,  $a\approx b$ indicates $a\lesssim b$ and $b\lesssim a$. 

\subsection{Dunford-Taylor Representation}

The function $u(s) \in L^2(\Omega)$ in  \eqref{def:prob} has the following representation \cite{MR1336382}
\begin{equation} \label{u:Dunford}
u(s) =\frac{1}{2\pi i}\int_{\mathcal{C}}z^{-s}(zI+\Delta)^{-1}fdz,
\end{equation}
where $\mathcal{C}$ is a Jordan curve oriented to have the spectrum of $-\Delta$ to its right. Deforming the contour $\mathcal C$ to the negative real axis, we obtain the Balakrishnan formula, valid for $s\in(0,1)$,
\begin{equation} \label{u:Balak}
u(s)=\frac{\sin(s\pi)}{\pi}\int_0^{\infty}\mu^{-s}(\mu I-\Delta)^{-1}fd\mu.
\end{equation}
The numerical integration of the above improper integral relies on a sinc quadrature method after the change of variable $y=\ln(\mu)$, leading to
\begin{equation} \label{def:u}
u(s)=\frac{\sin(s\pi)}{\pi}\int_{-\infty}^{\infty}e^{(1-s)y}(e^y I -\Delta)^{-1}fdy.
\end{equation}
	
\subsection{Finite Element Approximation}
We assume that $\Omega$ is a polyhedral domain and we consider a sequence $\{\Th\}_{h>0}$ of conforming and shape-regular partitions of $\Omega$ into $d$-simplices with maximal mesh size $h<1$. Let $\Vh$ be the  space of  continuous and piecewise linear finite element functions associated with $\Th$. The finite element approximation of \eqref{def:u} is then defined by
\begin{equation} \label{def:uh}
u_h(s):=\frac{\sin(s\pi)}{\pi}\int_{-\infty}^{\infty}e^{(1-s)y}w_h(y)dy,
\end{equation}
where $w_h(y)\in\Vh$ is the solution to
\begin{equation} \label{def:pb_wh}
a(w_h(y),v_h;y) = F(v_h), \quad \forall v_h\in\Vh.
\end{equation}
Here we used the notation 
\begin{equation}\label{def:pb_ay}
a(w,v;y) :=  a_0(w,v)+e^ya_1(w,v):= \int_{\Omega}\nabla w\cdot\nabla v +e^{y}\int_{\Omega}wv
\end{equation}
for $w,v\in H_0^1(\Omega)$ and
\begin{equation} \label{def:pb_F}
F(v)  :=  \int_{\Omega}fv,
\end{equation}
for $v\in H_0^1(\Omega)$.
The Poincar\'e inequality \eqref{def:Poincare} implies that for $v,w \in H^1_0(\Omega)$,
\begin{equation}\label{e:coerc_cont}
\| v \|_{H^1_0(\Omega)}^2 \leq a(v,v;y) \quad \textrm{and} \quad a(v,w;y) \leq (1+C_P^2e^y)\| v \|_{H^1_0(\Omega)}\| w \|_{H^1_0(\Omega)},
\end{equation}
which guarantees that \eqref{def:pb_wh} has a unique solution for any parameter $y \in \mathbb R$ by the Lax-Milgram lemma. 

We now collect some estimates for $w_h(y)$, which will be used in the analysis later.
\begin{lemma} \label{lem:tmp_res}
Let $C_P$ be the Poincar\'e constant in \eqref{def:Poincare}. For any $y,\bar y\in\mathbb{R}$, we have
\begin{equation} \label{apriori_wh}
\|\nabla w_h(y)\|\leq C_P\|f\|, \quad \|w_h(y)\|\leq e^{-y}\|f\| ,
\end{equation}
\begin{equation} \label{error_y_small}
\|\nabla (w_h(y)-w_h(\bar y))\|\leq C_P^3|e^y-e^{\bar y}|\|f\| ,
\end{equation}
\begin{equation} \label{error_y_large_1}
\|w_h(y)-w_h(\bar y)\|\leq e^{-y}|e^{y-\bar y}-1|\|f\|,
\end{equation}
and 
\begin{equation} \label{error_y_large_2}
\|\nabla (w_h(y)-w_h(\bar y))\|\leq \frac{1}{2}e^{-\frac{y}{2}}|e^{y-\bar y}-1|\|f\|.
\end{equation}
\end{lemma}
\begin{proof}
	Choosing $v_h=w_h(y)$ in \eqref{def:pb_wh} yields
	\begin{equation*}
	\|\nabla w_h(y)\|^2+e^y\|w_h(y)\|^2=\int_{\Omega}fw_h(y)\leq \|f\|\|w_h(y)\|,
	\end{equation*}
	from which the two relations in \eqref{apriori_wh} can be easily deduced. From \eqref{def:pb_wh} we get
	\begin{equation*}
	\int_{\Omega}\nabla(w_h(y)-w_h(\bar y))\cdot\nabla v_h+e^y\int_{\Omega}(w_h(y)-w_h(\bar  y))v_h = (e^{\bar y}-e^y)\int_{\Omega}w_h(\bar y)v_h \quad \forall v_h\in\Vh.
	\end{equation*}
	We now choose $v_h=w_h(y)-w_h(\bar y)$ to get
	\begin{equation} \label{error_y_tmp}
	\|\nabla(w_h(y)-w_h(\bar y))\|^2+e^y\|w_h(y)-w_h(\bar y)\|^2 \leq  |e^{\bar y}-e^y|\|w_h(\bar y)\|\|w_h(y)-w_h(\bar y)\|.
	\end{equation}
	This, the Poincar\'e inequality \eqref{def:Poincare} and \eqref{apriori_wh} with $y=\bar y$ yield \eqref{error_y_small}. 
	
	The estimate \eqref{error_y_large_1} follows from \eqref{error_y_tmp} together with \eqref{apriori_wh} with $y=\bar y$. For \eqref{error_y_large_2}, we invoke Young's inequality to estimate the right hand side of \eqref{error_y_tmp} and get
	\begin{equation*}
	\|\nabla(w_h(y)-w_h(\bar y))\| \leq  \frac{1}{2}e^{-\frac{y}{2}}|e^{\bar y}-e^y|\|w_h(\bar y)\|.
	\end{equation*}
	It remains to invoke (\ref{apriori_wh}) with $y=\bar y$ to derive the desired result and ends the proof.
\end{proof}

We mention that both results in (16) are standard, while the estimates (17), (18) and (19) are less common yet useful in the error analysis below. Moreover, note that \eqref{apriori_wh}-left and \eqref{error_y_small} are favorable for negative $y$, while \eqref{apriori_wh}-right, \eqref{error_y_large_1} and \eqref{error_y_large_2} for positive $y$. This plays a role in our analysis below and is observed in the numerical experiments, see Section \ref{sec:numres}.

We end this section by stating the error in the finite element method derived and analyzed in \cite{BP17}. Before doing this, we define $\alpha \in (0,1]$ to be the elliptic pick-up regularity index, i.e. $\alpha$ is the largest number in $(0,1]$ such that $(-\Delta)$ is an isomorphism from $\Hr$ to $\Hrp$ for all $r\in[0,\alpha]$. Notice that  $\alpha >0$ for Lipschitz domains and $\alpha=1$ when $\Omega$ is convex.
	
\begin{thm} \label{thm:FE_error}
Let $f\in L^2(\Omega)$, $\alpha>0$ denote the elliptic regularity pick-up and $\alpha^*:=\frac{\alpha+\min(\alpha,1-r)}{2}$. Then for any $r\in[0,1]$, we have
\begin{enumerate}[label=\arabic*.]
	\item If $r+2\alpha^*-2s\geq 0$ and $f\in \Ha$ then
		\begin{equation*} \label{eqn:FE_error_case1}
		\|u-u_h\|_{\Hr} \lesssim \ln(h^{-1})h^{2\alpha^*}\|f\|_{\Ha}.
		\end{equation*}
	\item If $r+2\alpha^*-2s\geq 0$ and $f\in \Hae$ with $r+2\alpha^*-2s+2\varepsilon\leq 1+\alpha$ then
		\begin{equation*} \label{eqn:FE_error_case2}
		\|u-u_h\|_{\Hr} \lesssim h^{2\alpha^*}\|f\|_{\Hae}.
		\end{equation*}
	\item If $r+2\alpha^*-2s<0$ then
		\begin{equation*} \label{eqn:FE_error_case3}
		\|u-u_h\|_{\Hr} \lesssim h^{2\alpha^*}\|f\|_{L^2(\Omega)}.
		\end{equation*}
	\end{enumerate}
\end{thm}

\subsection{Sinc quadrature approximation}
	
We discuss the sinc quadrature approximation leading to the fully discrete approximation $u_{h,k}$ given by \eqref{def:uk}.
Recall that $k>0$ is the sinc quadrature parameter and that $M_s$, $N_s$ are given by \eqref{def:NsMs}. This choice is dictated from the analysis of the sinc quadrature error, which is the subject of the following theorem; we refer to \cite{BLP18} for its proof. 
	
\begin{thm} \label{thm:sinc_error}
Let $f\in L^2(\Omega)$ and  $r\in[0,1]$. 
\begin{enumerate}[label=\arabic*.]
	\item If $s>r/2$ then
		\begin{equation*} \label{eqn:sinc_error_case1}
		\|u_h(s)-u_{h,k}(s)\|_{\Hr} \lesssim \left(e^{-\frac{\pi^2}{k}}+e^{-(1-s)M_sk}+e^{-sN_sk}\right)\|f\|_{L^2(\Omega)}.
		\end{equation*} \\
	\item If $s\leq r/2$ and $f\in \Hrse$ with $r-2s+\varepsilon\in[0,1+\alpha]$ then
		\begin{equation*} \label{eqn:sinc_error_case2}
		\|u_h(s)-u_{h,k}(s)\|_{\Hr} \lesssim \left(e^{-\frac{\pi^2}{k}}+e^{-(1-s)M_sk}+e^{-sN_sk}\right)\|f\|_{\Hrse}.
		\end{equation*}
	\end{enumerate}
\end{thm}

\section{Reduced Basis Approximation}\label{s:rb}
	
The computation of $u_{h,k}$ in \eqref{def:uk} involves the finite element solution $w_h(y)$ to the reaction diffusion problem \eqref{def:pb_wh} for $M_s+N_s+1$ different values of the parameter $y\in \I_s:=[-M_sk,N_sk]$. We propose to use the reduced basis method to approximate the entire map $y \mapsto w_h(y)$. Notice that the bilinear form $a(\cdot,\cdot;y)$ defining $w_h(y)$ in \eqref{def:pb_wh} is not affine in $y$. However, it becomes affine for $\mu=e^y$.

\subsection{Construction for a fixed $s$}\label{s:construction}
The reduced basis space
\begin{equation*}
\Vn = \Span\{w_h(y^1),...,w_h(y^n)\} \subset \mathbb V_h
\end{equation*}
is constructed using a greedy strategy \cite{BMPPT12}. Starting with $y^1=0$, $y^{m+1}   \in \I_s$ is selected iteratively to maximize the error
\begin{equation} \label{def:err_Wm}
W_m(y):=\|w_h(y)-P_{\mathbb V_h^{m}}w_h(y)\|_{H_0^1(\Omega)} 
\end{equation}
on $\I_s$, i.e.,
\begin{equation*}
y^{m+1}:=\argmax_{y\in\I_s}W_m(y). 
\end{equation*}
Here $P_{\mathbb V_h^m}w_h(y)\in \mathbb V_h^m$ is the unique solution (from Lax-Milgram theory) to
\begin{equation}\label{e:galerkin}
a(P_{\mathbb V_h^m}w_h(y),v_m;y) = a(w_h(y),v_m;y),  \quad \forall v_m\in\mathbb V_h^m.
\end{equation}
Notice that in view of the definition \eqref{def:pb_wh} of $w_h(y)$, the above relation is equivalent to
\begin{equation*}
a(P_{\mathbb V_h^m}w_h(y),v_m;y) = F(v_m),  \quad \forall v_m\in\mathbb V_h^m.
\end{equation*}
The enrichment of the reduced basis space ends when 
\begin{equation}\label{e:eps_W}
\max_{y \in \I_s} W_{m}(y) \leq \varepsilon \|f\|
\end{equation}
for a prescribed accuracy $\varepsilon>0$ or when a maximum number of basis functions $N_{\max}$ is reached. Note that the relations \eqref{error_y_small} and \eqref{error_y_large_2} guarantee that \eqref{e:eps_W} can always be achieved by a uniform selection of points $y$ in $\mathcal D_s$. The aim of the Greedy algorithm is to provide an alternate selection performing as well but of less cardinality.
	
The error $W_{m}(y)$ defined in \eqref{def:err_Wm} is not a computable quantity and is usually replaced by an equivalent computable quantity leading to the so-called \emph{weak greedy algorithm} \cite{DPW13}. In this work, we use the residual based \emph{a posteriori} error estimate \cite{PR06}
\begin{equation} \label{def:apost}
\|r_{m}(\cdot;y)\|_{\Vh'}:=\sup_{v_h\in \Vh}\frac{r_{m}(v_h;y)}{\|v\|_{H_0^1(\Omega)}},
\end{equation}
where
$$r_{m}(v_h;y):=F(v_h)-a(P_{\mathbb V_h^{m}}w_h(y),v_h;y).$$
We have the following equivalence relation between the error $W_{m}(y)$ and its surrogate
\begin{equation} \label{eqn:equiv_err_est}
(1+C_P^2e^y)^{-1}\|r_m(\cdot;y)\|_{\Vh'} \leq W_m(y)\leq \|r_m(\cdot;y)\|_{\Vh'}
\end{equation}
for all $y\in \mathbb R$. This follows from
\begin{equation*}
\|r_m(\cdot;y)\|_{\Vh'}=\sup_{v_h\in\Vh}\frac{a(w_h(y)-P_{\mathbb V_h^{m}}w_h(y),v_h;y)}{\|v_h\|_{H_0^1(\Omega)}},
\end{equation*}
where $w_h(y)$ satisfies \eqref{def:pb_wh}, and the coercivity and continuity of the bilinear form $a$  \eqref{e:coerc_cont}.

Hence, selecting the samples using $\|r_{m}(\cdot;y)\|_{\Vh'}$ as surrogate for the error $W_{m}(y)$ yields
\begin{equation*}
W_m(y^{m+1})\ge \gamma_s\max_{y\in \I_s} W_m(y),
\end{equation*}
where
\begin{equation}\label{e:gamma_s}
\gamma_s:=\left(\max_{y\in \I_s} (1+C_P^2e^y)\right)^{-1}= (1+C_P^2e^{N_sk})^{-1}.
\end{equation}
The parameter $\gamma_s$ corresponds to the constant in the weak greedy algorithm, see \cite{DPW13} for more details, and will appear in the analysis below.

The dual norm can be computed using the Riesz representation theorem, see for instance \cite{PR06,EPR10} for more details. However, evaluating $\|r_{m}(\cdot;y)\|_{\Vh'}$ for every $y$ in $\I_s$ remains unfeasible. In practice, the maximization is performed over a finite dimensional \emph{training set} $\Theta_s\subset \I_s$, either chosen sufficiently fine to retain the performance of the algorithm, see for instance \cite{CD15}, or based on a random selection of moderate size \cite{CDD18}.

The reduced basis space is constructed \emph{offline} and gives the following \emph{online} approximation of $u_{h,k}(s)$
\begin{equation} \label{def:un}
u_{h,k}^n(s)=\frac{k\sin(s\pi)}{\pi}\sum_{l=-M_s}^{N_s}e^{(1-s)y_l}w_h^n(y_l), \quad w_h^n(y_l):=P_{\Vn}w_{h}(y_l).
\end{equation}

\begin{remark}
The approximate solution \eqref{def:uk} obtained without using the reduced basis method requires to solve $M_s+N_s+1$ sparse finite element systems of dimension $N_h$. In comparison, the solution of $M_s+N_s+1$ systems are needed for the approximation \eqref{def:un}, each requiring $\mathcal{O}(n^3)$ operations, where $n$ stands for the dimension of the reduced basis space. The latter is built once for all (\emph{offline stage}) with a computational cost dominated by the solve of $n$ sparse finite element systems. The value of $n$ depends on the Kolmogorov $n$-width of the solution manifold $\{w_h(y): \, y\in\mathcal{D}_s\}$. We refer for instance to \cite{QMN2016,CD15} for a detailed complexity analysis of the reduced basis method but note that typically for elliptic problems we have $n <\!\!< N_h$. Finally, we anticipate that the proposed reduced basis space is independent on $s$, see Section \ref{sec:universal}.
\end{remark}	

\subsection{Error analysis for a fixed $s$}\label{s:error_s}
	
We now analyze the distortion between $u(s)=(-\Delta)^{-s}f$ and its reduced basis approximation $u_{h,k}^n(s)$ given by \eqref{def:un} in the $\Hr$ norm. In order to avoid unnecessary technicalities, we assume from now on that $r\in[0,1]$ is chosen such that $r<2s$, which includes the natural choice $r=s$ leading to the energy error and $r=0$ for the $L^2(\Omega)$ error. The discussion below can be readily extended to the case $r\geq 2s$ by accounting for the log factor $\ln(h^{-1})$ in the finite element approximation (see Theorem~\ref{thm:FE_error}). 
For $u(s)\in \Hr$, we decompose the error into three parts
\begin{equation}\label{e:error_decom}
\begin{split}
	\|u(s)-u_{h,k}^n(s)\|_{\Hr}&\leq \|u(s)-u_h(s)\|_{\Hr}+\|u_h(s)-u_{h,k}(s)\|_{\Hr}\\
	& \qquad +\|u_{h,k}(s)-u_{h,k}^n(s)\|_{\Hr},
\end{split}
\end{equation}
corresponding to the finite element error, the sinc quadrature error and the reduced basis error, respectively. 
Given a target tolerance $\varepsilon>0$, we construct a reduced basis space such that \eqref{e:eps_W} holds.  In view of Theorems \ref{thm:FE_error} and \ref{thm:sinc_error}, we select the space discretization and sinc quadrature parameters $h$ and $k$ to balance the finite element and sinc quadrature errors, i.e.,
\begin{equation} \label{eqn:scaling_h_k}
C_{\textrm{FEM}} h^{2\alpha^*} = C_{\textrm{SINC}} e^{-\frac{\pi^2}{k}} = \varepsilon
\end{equation}
for some absolute constants $C_{\textrm{FEM}}$ and $C_{\textrm{SINC}}$.
	
We now assess the error in the reduced basis modeling by analyzing the behavior of $\|u_{h,k}(s)-u_{h,k}^n(s)\|_{\Hr}$ as $n$ increases. From the definitions \eqref{def:uk} and \eqref{def:un} of $u_{h,k}(s)$ and $u_{h,k}^n(s)$, respectively, we have
\begin{equation}\label{e:Es}
\|u_{h,k}(s)-u_{h,k}^n(s)\|_{\Hr} \leq \frac{k\sin(s\pi)}{\pi}\sum_{l=-M_s}^{N_s}e^{(1-s)y_l}\|w_h(y_l)-w_h^n(y_l)\|_{\Hr}.
\end{equation}

Key ingredients in our analysis are estimates for the reduced basis errors $\|w_h(y_l)-w_h^n(y_l)\|_{\Hr}$ in approximating the inner problems. We discuss this now. Recall that the reduced basis error for the inner problem is given by
\begin{equation}\label{e:error_rb_def}
\sup_{y \in \I_s}\| \nabla(w_h(y)-w_h^n(y)) \|=\sup_{y \in \I_s}\| \nabla(w_h(y)-P_{\Vn}w_h(y)) \|.
\end{equation}
The Kolmogorov $n$-width 
\begin{equation} \label{eqn:Kolmo}
d_n:=\inf_{\dim(Y_n)\le n} \,\, \sup_{y\in \I_s} \,\, \inf_{v_n\in Y_n} \| \nabla(w_h(y)-v_n)\|, \quad n\ge 1,
\end{equation}
is the benchmark for the best achievable decay. By convention, we set
\begin{equation} \label{eqn:Kolmo_0}
d_0:=\sup_{y\in \I_s} \| \nabla w_h(y)\|.
\end{equation}
It is quite remarkable that the linear space $\Vn$ constructed by the (weak) greedy selection discussed in Section~\ref{s:construction} leads to an error \eqref{e:error_rb_def} equivalent to $d_n$ \cite{DPW13}. In particular, an exponential decay of the Kolmogorov $n$-width guarantees an exponential decay of the (weak) greedy error. In order to prove that the error in (\ref{e:error_rb_def}) decays exponentially with $n$, see Lemma \ref{lem:error_RB_w} below, we will thus show that the Kolmogorov $n$-width exhibits an exponential decay.

To facilitate the analysis of $d_n$, we use the notations in \eqref{def:pb_ay} and provide a representation of the finite element functions $w_h(y)$ in term of the eigenpairs $\{\mu_i,\varphi_i\}_{i=1}^{N_h} \subset \mathbb{R}^+ \times\Vh$, $N_h:=\textrm{dim}(\Vh)$, of the generalized eigenvalue problem
\begin{equation*}
a_1(\varphi_i,v_h) = \mu_i a_0(\varphi_i,v_h), \qquad \forall v_h \in \Vh.
\end{equation*}
Without loss of generality, we assume that the $\varphi_i$ are $H_{0}^1$-orthonormal, i.e.
\begin{equation*}
a_0(\varphi_i,\varphi_j) = \delta_{ij}, \qquad 1\leq i,j\leq N_h.
\end{equation*}
The inverse inequality
\begin{equation*}
\|\nabla v_h\| \leq C_I h^{-1}\|v_h\| \quad \forall v_h\in\Vh,
\end{equation*}
together with the Poincar\'e inequality \eqref{def:Poincare}, yield
\begin{equation} \label{eqn:mu}
C_I^{-2} h^2 \leq \mu_i \leq C_P^2, \quad 1\leq i\leq N_h.
\end{equation}
With these notations, we can rewrite $w_h(y)$ in (\ref{def:pb_wh}) as
\begin{equation}\label{evrepre}
w_h(y) = \sum_{i = 1}^{N_h}\frac{f_i\varphi_i}{1 + e^y \mu_i}, \quad f_i:=F(\varphi_i).
\end{equation}

We are now in position to assess the reduced basis approximation property. 
\begin{lemma} \label{lem:error_RB_w}
For any $n\ge 1$ we have
\begin{equation} \label{eqn:exp_error_RB_w}
\sup_{y \in \I_s}\|\nabla(w_h(y)-P_{\Vn}w_h(y))\| \leq \gamma_s^{-1}C_1 e^{-C_2(h)n}\|f\|,
\end{equation}
where $\gamma_s$ is given by \eqref{e:gamma_s}, $C_1$ is a constant only depending on $C_P$ and 
\begin{equation} \label{eqn:cst_C2}
C_2(h)\approx\frac{1}{\ln(C_P^2C_I^2h^{-2})} \qquad \textrm{when }h\to 0.
\end{equation}
\end{lemma}
\begin{proof}
We follow \cite{DS19} to construct a linear space $\Wn\subset\Vh$ with $\dim(\Wn)\leq n$ such that for some constants $c_1$, $c_2$ and $n\geq 1$ we have
\begin{equation} \label{eqn:step1}
d_n \leq \sup_{y \in \I_s}\inf_{v_h^n\in \Wn}\| \nabla(w_h(y)-v_h^n)\|\le c_1 e^{-c_2n}\| f\|.
\end{equation}
Let $\Wn:=\Span\{w_h(y_1),\ldots,w_h(y_n)\}$, where the $y_j$ are chosen such that the $e^{y_j}$ are the transformed Zolotar\"ev points $[C_P^{-2},C_I^2h^{-2}]$ as in \cite{DS19}, see also \cite{gonvcar1969zolotarev,MR1328645}. Notice that this interval is dictated by the lower and upper bound of the eigenvalues $\mu_i$, see \eqref{eqn:mu}. Now, given $y\in\I_s$, we define the approximation
\begin{equation} \label{eqn:vhn}
v_h^n({y}) := \sum_{j=1}^n\alpha_j(y)w_h(y_j)\in\Wn,
\end{equation}
where the coefficients $\alpha_j(y)$ are such that
\begin{equation*}
\frac{1}{1+e^{y}e^{-y_k}} = \sum_{j=1}^n\alpha_j(y)\frac{1}{1+e^{y_j}e^{-y_k}}, \qquad k=1,..,n.
\end{equation*}
The above system is a particular rational interpolation problem and has a unique solution according to Lemma 5.13 in \cite{DS19}.
Furthermore, thanks to Lemma 5.17 in \cite{DS19}, we have
\begin{equation*}
\left|\frac{1}{1+e^{y}\mu_i}-\sum_{j=1}^n\alpha_j(y)\frac{1}{1+e^{y_j}\mu_i}\right|\lesssim \frac{1}{1+e^{y}\mu_i}e^{-C^*n}  ,  \quad i=1,\ldots, N_h, 
\end{equation*}
where
\begin{equation*}
C^*=C^*(h)\approx\frac{1}{\ln(C_P^2C_I^2h^{-2})}.
\end{equation*}
Hence, the error between $w_h(y)$ in \eqref{evrepre} and $v_h^n({y})$ in \eqref{eqn:vhn} satisfies
\begin{eqnarray*} \label{eqn:err_Z2}
\|w_h(y)-v_h^n({y})\|_{H_0^1}^2 & = &  \sum_{i=1}^{N_h}f_i^2\left(\frac{1}{1+e^y\mu_i}-\sum_{j=1}^n\alpha_j(y)\frac{1}{1+e^{y_j}\mu_i}\right)^2 \nonumber \\
 & \lesssim & e^{-2C^*n}\sum_{i=1}^{N_h}f_i^2\left(\frac{1}{1+e^{y}\mu_i}\right)^2 \nonumber \\
 & \lesssim & e^{-2C^*n}\|\nabla w_h(y)\|^2,
\end{eqnarray*}
where we have used the $H_{0}^1$-orthonormality of the $\{\varphi_i\}_{i=1}^{N_h}$. 
With the help of  \eqref{apriori_wh}, this implies
\begin{equation} \label{eqn:err_Z}
\|w_h(y)-v_h^n({y})\|_{H_0^1}^2 \lesssim e^{-2C^*n}C_P^2\|f\|^2.
\end{equation}
The above estimate is \eqref{eqn:step1} with $c_2=C^*$ and $c_1$ only depending on $C_P$ and the hidden constant in \eqref{eqn:err_Z}. Moreover, thanks to \eqref{apriori_wh} we also have $d_0\leq C_P\|f\|$, where $d_0$ is defined in \eqref{eqn:Kolmo_0}. Therefore, we have shown that the Kolmogorov $n$-width (see \eqref{eqn:Kolmo} and \eqref{eqn:Kolmo_0}) satisfies
\begin{equation} \label{eqn:step2}
d_n \leq c_1e^{-c_2n} \| f\|, \quad n\ge 0.
\end{equation}
	
To conclude, it remains to relate the error decay of the reduced basis generated by the weak greedy algorithm with parameter $\gamma_s$ (see $\eqref{e:gamma_s}$) and the Kolmogorov $n$-width $d_n$. Corollary 8.4 in \cite{CD15}, see also Corollary 3.3 in \cite{DPW13}, guarantees that (\ref{eqn:step2}) implies \eqref{eqn:exp_error_RB_w} with $C_2=c_2/6=C^*/6$ and $C_1=c_1\max(\sqrt{2},\gamma_s e^{C_2})\lesssim c_1$.
\end{proof}

\begin{remark} \label{rem:Maday}
We mention that an exponential decay for the reduced basis error for one dimensional parametric problem of the form (\ref{def:pb_wh}) has already been obtained in \cite{maday2002priori,MPT02}. However, the exponential decay is guaranteed for $n \geq n_\textrm{crit}$ for some integer $n_\textrm{crit}$ depending on the length of the parameter interval $[-M_s k,N_s k]$. We did not pursue this route as the latter restriction seems prohibitive to take full advantage of the performances of the reduced basis method. 
\end{remark}
	
We now use the exponential decay of the reduced basis error for $w_h$ obtained in Lemma \ref{lem:error_RB_w} to estimate the error for $u_{h,k}$ defined in \eqref{e:Es}.
\begin{lemma} \label{lem:E1}
Let $h$ and $k$ be the finite element and sinc quadrature parameters.
Let $r\in[0,1]$. For  $n \geq 1$, we have
\begin{equation} \label{eqn:exp_error_RB_u}
\begin{split}
&\|u_{h,k}(s)-u_{h,k}^n(s)\|_{\Hr} \\
& \qquad \leq \frac{\sin(s\pi)}{(1-s)\pi}C_P^{1-r}\gamma_s^{-1} C_1 e^{-C_2(h)n}\left(e^{(1-s)N_s k}-e^{-(1-s)M_sk}\right)\|f\|,
\end{split}
\end{equation}
where $C_1$ and $C_2(h)$ are the constants  in \eqref{eqn:exp_error_RB_w}.
\end{lemma}
\begin{proof}
Because $\Hr$ are interpolation spaces between $L^2(\Omega)$ and $H^1_0(\Omega)$, see \eqref{def:Hr}, the Poincar\'e inequality \eqref{def:Poincare} yields
\begin{equation*}
\begin{split}
\|w_h(y)-w_h^n(y)\|_{\Hr}& \leq \|w_h(y)-w_h^n(y)\|^{1-r} \| \nabla(w_h(y)-w_h^n(y)) \|^r \\
&\leq C_P^{1-r}\|\nabla(w_h(y)-w_h^n(y))\|.
\end{split}
\end{equation*}
Therefore, using the estimate \eqref{e:Es} for the error and invoking Lemma~\ref{lem:error_RB_w}, we  get
\begin{equation*}
\begin{split}
&\|u_{h,k}(s)-u_{h,k}^n(s)\|_{\Hr} \\
& \qquad \leq  \frac{k\sin(s\pi)}{\pi}C_P^{1-r} \gamma_s^{-1} C_1 e^{-C_2(h) n}\|f\|\sum_{l=-M_sk}^{N_s k}e^{(1-s)y_l} \\
& \qquad\leq  \frac{\sin(s\pi)}{\pi}C_P^{1-r}\gamma_s^{-1} C_1 e^{-C_2(h)n}\|f\| \int_{-M_sk}^{N_sk}e^{(1-s)y}dy \\
& \qquad\leq  \frac{\sin(s\pi)}{(1-s)\pi}C_P^{1-r}\gamma_s^{-1} C_1 e^{-C_2(h)n}\left(e^{(1-s)N_s k}-e^{-(1-s)M_sk}\right)\|f\|,
\end{split}
\end{equation*}	
which is the claimed estimate.
\end{proof}

\begin{remark} \label{rem:expo_decay}
From \eqref{eqn:cst_C2}, we see that the constant $C_2(h)$ that appears in \eqref{eqn:exp_error_RB_w} and \eqref{eqn:exp_error_RB_u} tends to $0$ as $h$ tends to $0$. In other words, the reduced basis performances deteriorate as the target accuracy $\varepsilon$ tends to $0$. This phenomenon is observed in the numerical experiments reported in Figure \ref{fig:err_u_wrt_h} of Section \ref{sec:numres}.
\end{remark}
	
Using Lemma \ref{lem:E1}, we directly derive the following result providing a sufficient condition on the dimension $n$ of the reduced space to achieve an error under a specified tolerance $\delta>0$.
	
\begin{thm} \label{cor:err_Ds1}(offline construction of the reduced space)
Let $\varepsilon > 0$ be a given tolerance.
Assume that $h$ and $k$ are chosen so that \eqref{eqn:scaling_h_k} holds and that the reduced basis space is constructed such that \eqref{e:eps_W} holds. 
Then, for any $\delta \geq \varepsilon$ we have
$$
\|u_{h,k}(s)-u_{h,k}^n(s)\|_{\Hr}  \leq \delta \| f\|
$$
provided 
\begin{equation} \label{eqn:n_star}
n \approx \ln( C \delta \varepsilon^{\frac{2-s}{s}}) \ln(\varepsilon^{1/\alpha^*}),
\end{equation}
where $C$ is a constant only depending on $s$ and $r$ and $\alpha^*$ is as in Theorem~\ref{thm:FE_error}.
In particular
$$
n \approx  \ln(\varepsilon)^2
$$
when $\delta=\varepsilon$.
\end{thm}
\begin{proof}
The claims directly follow from Lemma~\ref{lem:E1} together with \eqref{eqn:scaling_h_k}.
\end{proof}

\subsection{Universal Reduced Basis Space} \label{sec:universal}
In the previous section we constructed a reduced basis space $\Vn$ to approximate $u_{h,k}(s)$ for a fixed $s \in (0,1)$. We now show that it is possible to take real advantage of the \emph{offline} work and construct reduced basis spaces approximating the map $s \mapsto u_{h,k}(s)$ for $s \in [s_{\min},s_{\max}]$, with $0<s_{\min} \leq s_{\max} <1$ fixed.

To see this, it suffices to adjust the constant depending on $s$ as follows. First, we let 
\begin{equation} \label{def:NM}
M:=\left\lceil \frac{\pi^2}{(1-s_{\max})k^2} \right\rceil \quad \mbox{and} \quad N:=\left\lceil \frac{\pi^2}{s_{\min}k^2} \right\rceil.
\end{equation}
Then, we define the domain $\I:=[-Mk,Nk]$ containing $\I_{s}$ for all $s \in [s_{\min},s_{\max}]$ and, similarly to \eqref{e:gamma_s}, we introduce the parameter
\begin{equation}\label{e:gamma}
\gamma:=\left(\max_{y\in \I} (1+C_Pe^y)\right)^{-1}= (1+C_Pe^{Nk})^{-1}.
\end{equation}
Finally, for all $y\in\I$ we approximate $w_h(y)$ by $w_h^n(y)=P_{\Vn}w_h(y)$, where the reduced basis space $\Vn$ is constructed as detailed in Section~\ref{s:construction} upon replacing $M_s$, $N_s$, $\I_s$ and $\gamma_s$ by $M$, $N$, $\I$ and $\gamma$, respectively. 
 
With this uniform construction, we directly obtain the universal version of Lemma~\ref{lem:E1} and Theorem~\ref{cor:err_Ds1}.
\begin{lemma} \label{lem:E2}
Let $h$ and $k$ be the finite element and sinc quadrature parameters. Let $r\in[0,1]$. For  $n \geq 1$ and any $s\in[s_{\min},s_{\max}]$ we have
\begin{equation*}
\|u_{h,k}(s)-u_{h,k}^n(s)\|_{\Hr} \leq \frac{\sin(s\pi)}{(1-s)\pi}C_P^{1-r}\gamma^{-1} C_1 e^{-C_2(h)n}\left(e^{(1-s)N k}-e^{-(1-s)Mk}\right)\|f\|,
\end{equation*}
where $C_1$ and $C_2(h)$ are the constants  in \eqref{eqn:exp_error_RB_w}.
\end{lemma}
\begin{thm} \label{cor:err_D}(offline construction of the universal reduced space)
Let $\varepsilon > 0$ be a given tolerance.
Assume that $h$ and $k$ are chosen so that \eqref{eqn:scaling_h_k} holds and that the reduced basis space is constructed such that \eqref{e:eps_W} holds. 
Then, for any $\delta \geq \varepsilon$ we have
\begin{equation*}
\max_{s\in[s_{\min},s_{\max}]}\|u_{h,k}(s)-u_{h,k}^n(s)\|_{\Hr}  \leq \delta \|f\|
\end{equation*}
provided 
\begin{equation} \label{eqn:n_star_univ}
n \approx \ln( C \delta \varepsilon^{\frac{2-s_{\min}}{s_{\min}}}) \ln(\varepsilon^{1/\alpha^*}),
\end{equation}
where $C$ is a constant only depending on $s_{\min}$, $s_{\max}$ and $r$, and $\alpha^*$ is as in Theorem~\ref{thm:FE_error}. In particular
\begin{equation*}
	n \approx  \ln(\varepsilon)^2
\end{equation*}
when $\delta=\varepsilon$.
\end{thm}

\section{Numerical Experiments} \label{sec:numres}
	
We present numerical results to illustrate the performances of the reduced basis approach analyzed in the previous section. Since the focus of this paper is on the reduced basis approximation, the finite element meshsize $h$ and the sinc quadrature parameter $k$ are chosen sufficiently small not to influence the total error, unless otherwise specified. We refer to \cite{BP15,BLP18} for an extensive numerical study  on the influence of the discretization parameters $h$ and $k$. The space $\Vn$ is built using a weak greedy algorithm on $[-Mk,Nk]$, starting with the \emph{snapshot} $w_h(0)$. Moreover, the \emph{training set} $\Theta$ consists of $10000$ uniformly distributed points in $[-Mk,Nk]$. Finally, we set $f=1$ in all the numerical examples.
	
\subsection{1D example} \label{sec:numres_1D}
	
We consider the case $\Omega=(0,1)$. The subdivision $\Th$ of $\Omega$ consists of a uniform partition of $[0,1]$ with subintervals of length $h=2^{-12}$. The sinc quadrature parameter is fixed to $k=0.5$ and the fractional power $s$ varies from $s_{\min}=0.1$ to $s_{\max}=0.9$. In this setting, we have $N=M=395$ and $[-Mk,Nk]=[-197.5,197.5]$.
	
\subsubsection{Reduced basis error for $w_h$}

We provide in Figure \ref{fig:err_w}-left the evolution of 
\begin{equation*}
e_w(n):=\sup_{y_l \in \Theta} \|w_h(y_l)-P_{\Vn}w_h(y_l)\|_{H_0^1(\Omega)}
\end{equation*}
versus $n\geq 1$ as indicator of the error
\begin{equation*}
\sup_{y \in [-Mk,Nk]} \|w_h(y_l)-P_{\Vn}w_h(y_l)\|_{H_0^1(\Omega)}.
\end{equation*}
The observed exponential decay matches the estimate of Lemma~\ref{lem:error_RB_w}.
	
Moreover, Figure \ref{fig:err_w}-right reports the values of the selected parameters $y^n$ by the weak greedy procedure. We observe that except for $y^2$, they are all located in the interval $[0,20]$. This behavior can be in part explained by the estimates provided in Lemma~\ref{lem:tmp_res}, indicating the robustness of $w_h(y)$ for small values of $y$ and the smallness of $\|w_h(y)\|$ for $y$ large. The fact that no negative $y$ is selected is also attributable to the choice of the initial snapshot, namely $w_h(0)$, which already provides a good approximation of $w_h(y)$ for $y\leq 0$. We mention that similar results are obtained when changing the range for $s$, for instance setting $s_{\min}=0.01$ and $s_{\max}=0.99$.

\begin{figure}[htbp]
\centering
\includegraphics[width=0.45\textwidth]{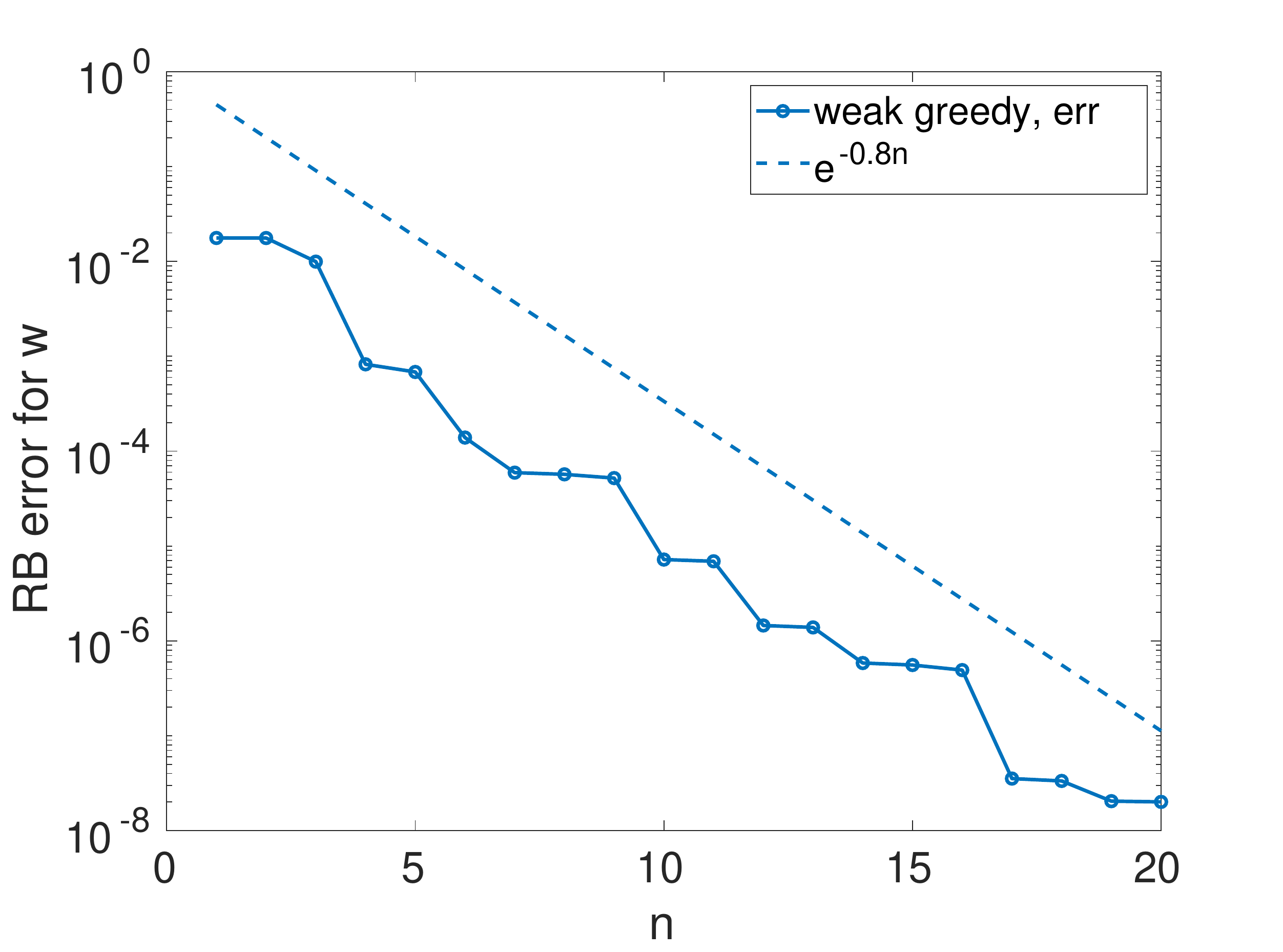}
\includegraphics[width=0.45\textwidth]{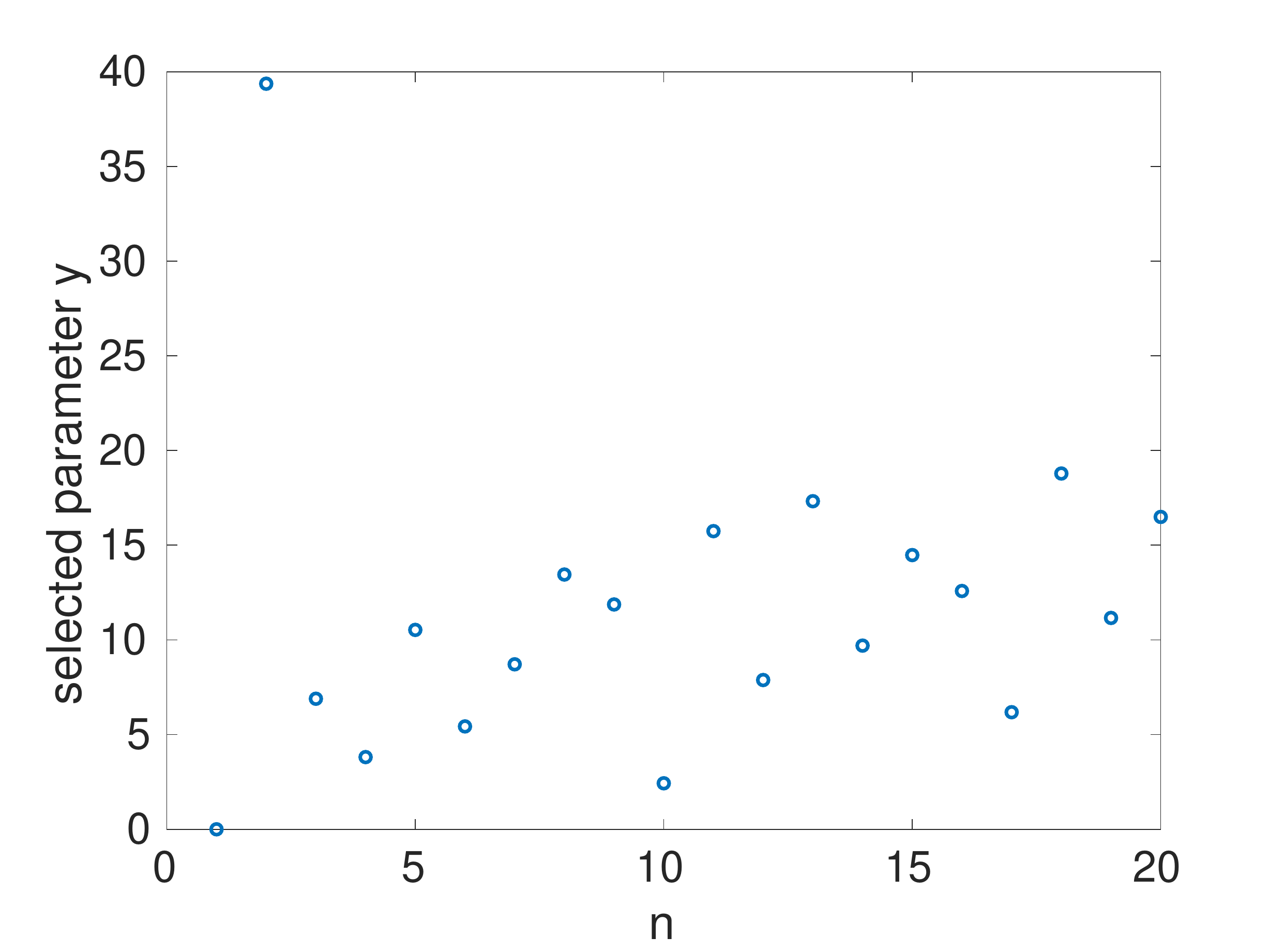}
\caption{Left: reduced basis error $e_w(n)$ versus $n$. Right: parameters $y^n$ selected by the weak greedy procedure during the construction of the reduced basis space.} \label{fig:err_w}
\end{figure}

We comment on the use of an \emph{a posteriori} error estimate (weak greedy) in place of the \emph{true} error (greedy), see \eqref{eqn:equiv_err_est}. For this, we compare in Figure \ref{fig:err_w_comparison} the performance of both algorithms and we can conclude that very little efficiency is lost in using the computable \emph{a posteriori} error estimator.

\begin{figure}[htbp]
\centering
\includegraphics[width=0.45\textwidth]{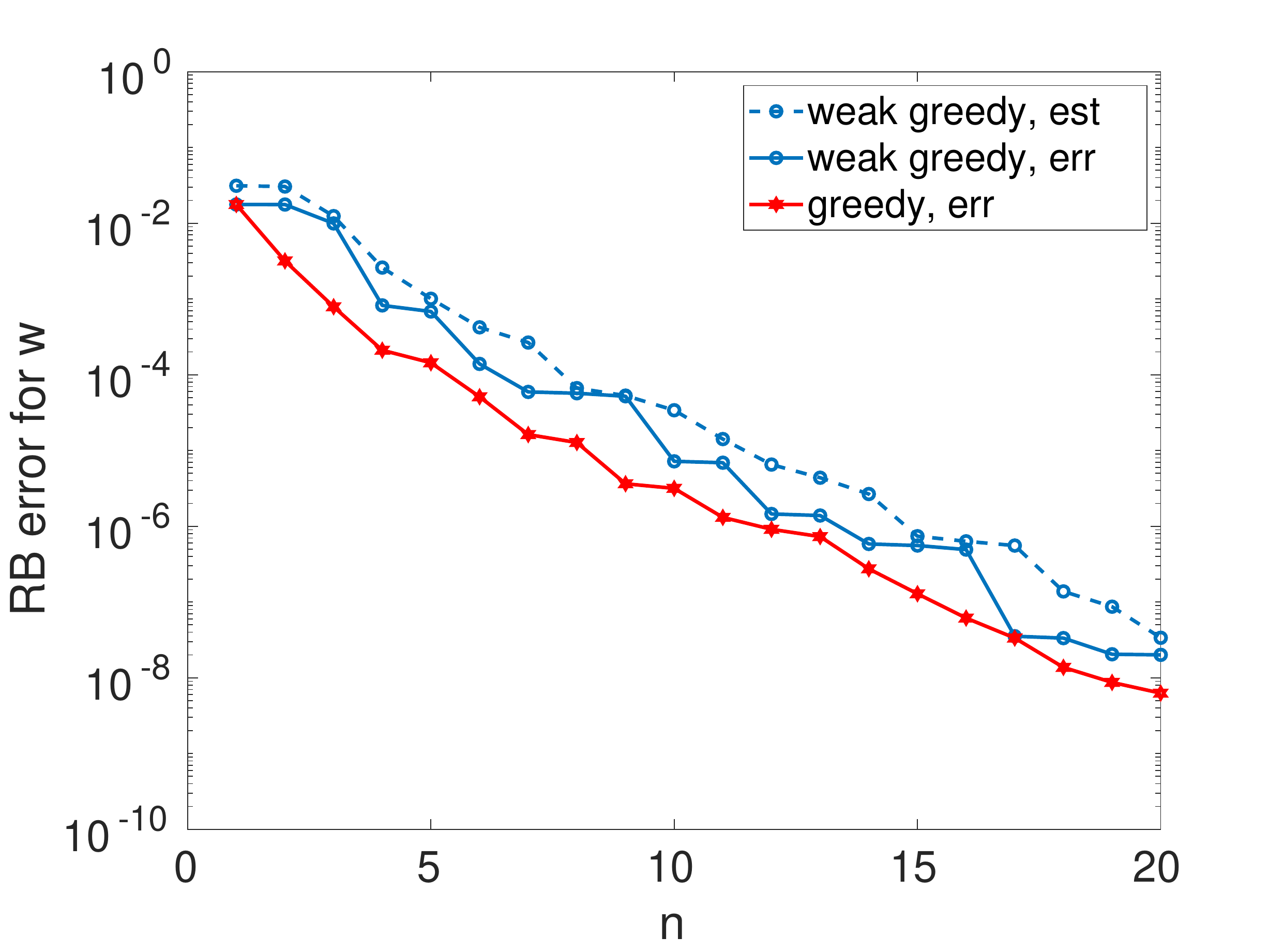}
\includegraphics[width=0.45\textwidth]{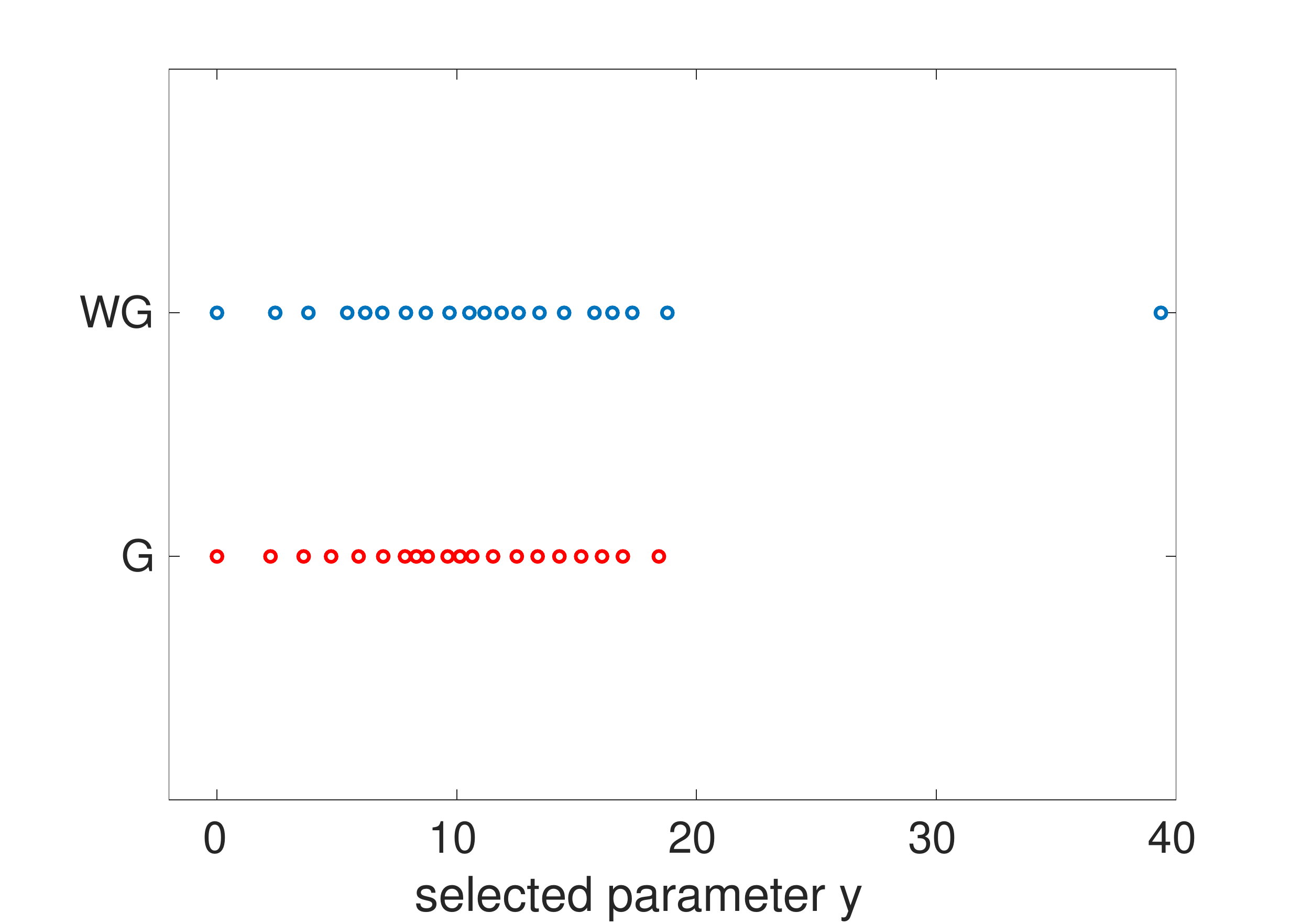}
\caption{Comparison of the greedy and weak greedy strategies. Left: error $e_w(n)$ associated with the greedy and weak greedy strategies. The dashed line represents the equivalent quantity $\max_{y\in\Theta}\|r_n(\cdot;y)\|_{\Vh'}$, see \eqref{eqn:equiv_err_est}. Right: selected parameters $y$ for both strategies.} \label{fig:err_w_comparison}
\end{figure}
	
\subsubsection{Reduced basis error for $u_{h,k}$} 

We now turn our attention to the approximation of $u_{h,k}(s)$ by $u_{h,k}^n(s)$. Figure \ref{fig:err_u} depicts the evolution of 
\begin{equation*}
e_{u(s)}(n):=\|u_{h,k}(s)-u_{h,k}^n(s)\|
\end{equation*}
for various values of fractional power. In agreement with Theorem~\ref{cor:err_D}, exponential decay is observed in all cases when using the universal reduced basis  space. Notice that in this experiment the sinc quadrature requires $440$ points for $s=0.1,0.9$, $190$ points  for $s=0.3,0.7$ and $159$ points for $s=0.5$ to guarantee a sinc quadrature error of the order $e^{-\pi^2/k}\approx 2.7\times 10^{-9}$ for $k=0.5$. A comparable reduced basis accuracy in the $L^2(\Omega)$ norm is achieved for only $n=20$ for $0.5\leq s\leq 0.9$.
\begin{figure}[htbp]
\centering
\includegraphics[width=0.45\textwidth]{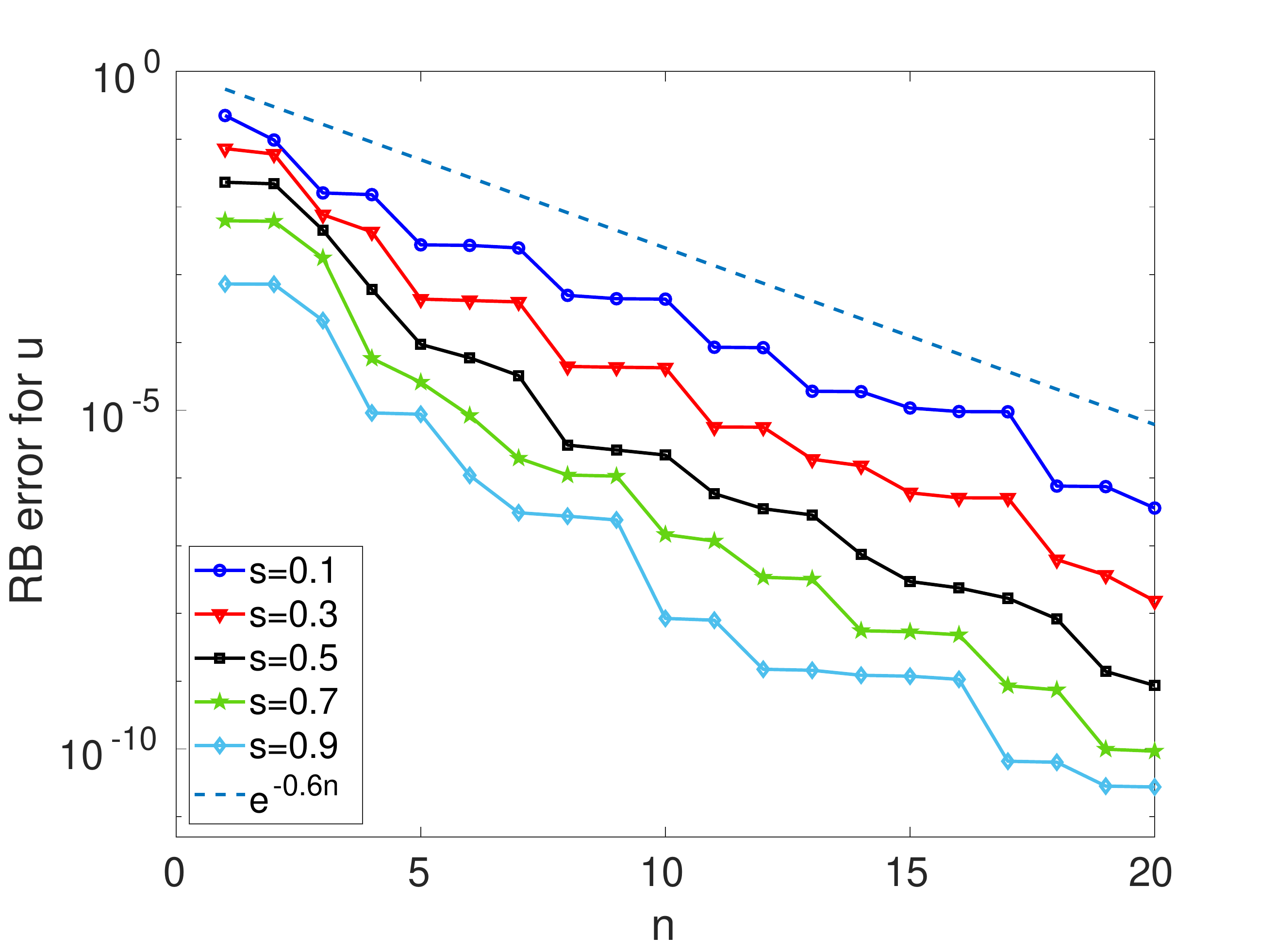}
\caption{Error $e_{u(s)}(n)$ with respect to $n$ for various values of $s$ in $[0.1,0.9]$. Exponential decay is observed in all cases using the universal reduced basis space.} \label{fig:err_u}	
\end{figure}
	
Finally, we study numerically the behavior of the constant $C_2(h)$ given by (\ref{eqn:cst_C2}). We set $s=0.1$ and consider a sequence of uniform partitions of $[0,1]$ with subintervals of length $h=2^{-k}$ for $k=6,8,10,12$. The reduced basis error $e_{u(0.1)}(n)$ versus $n$ is reported in Figure \ref{fig:err_u_wrt_h} for each finite element discretization. As predicted by Theorem~\ref{cor:err_D}, the exponential decay encoded in $C_2 = C_2(h)$ deteriorates as $h \to 0$. The $L^2(\Omega)$ norm of the error behaves like $e^{-1.4n}$ for $h=2^{-6}$ and $e^{-0.7n}$ for $h=2^{-12}$.
		
\begin{figure}[htbp]
\centering
\includegraphics[width=0.45\textwidth]{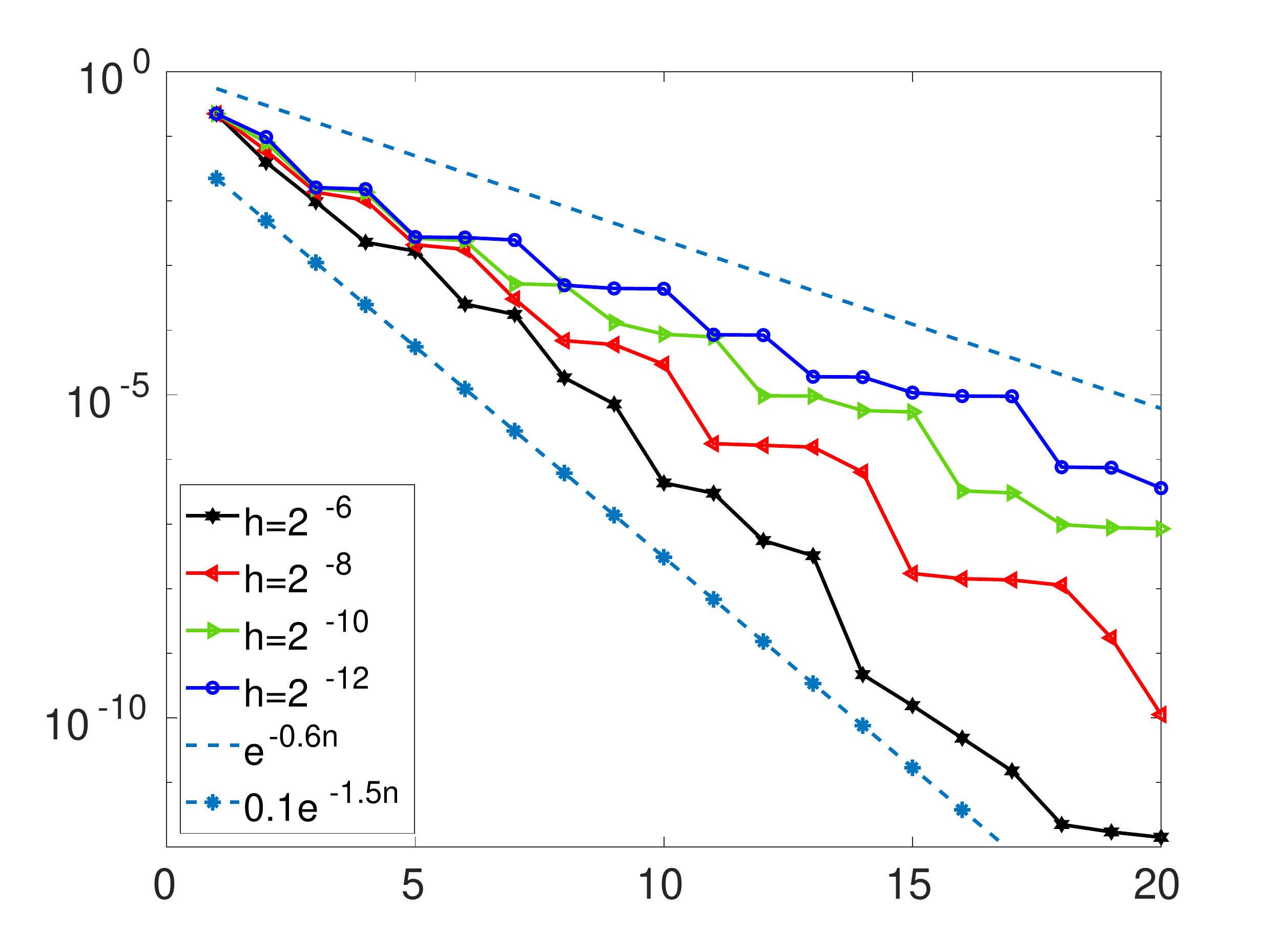}
\caption{Effect of the space discretization parameter $h$ in the exponential decay of the error $e_{u(0.1)}(n)$. In accordance with Theorem~\ref{cor:err_D}, the exponential decay coefficient deteriorates as $h$ decreases. } \label{fig:err_u_wrt_h}	
\end{figure}
	
\subsection{2D examples}
	
We now consider two dimensional domains: a square domain $\Omega=(0,1)^2$ and an L-shaped domain $\Omega = (0,1)^2\setminus ([0,0.5]\times[0.5,1])$. The space discretization consists of a Delaunay triangulation with $22968$ elements for $\Omega=(0,1)^2$ and $17190$ elements for $\Omega = (0,1)^2\setminus ([0,0.5]\times[0.5,1])$. In both cases, the elements in the triangulation have diameters between $0.005$ and $0.01$. All the other parameters are the same as in Section \ref{sec:numres_1D}. 
	
The evolution of the reduced basis error $e_{u(s)}(n)$ for different values of $s$ is reported in Figure \ref{fig:err_u_2D}. As for the one dimensional case, exponential decays are observed for all values of $s \in [0.1,0.9]$ and irrespectively of the shape of the domain.
		
\begin{figure}[htbp]
\centering
\includegraphics[width=0.45\textwidth]{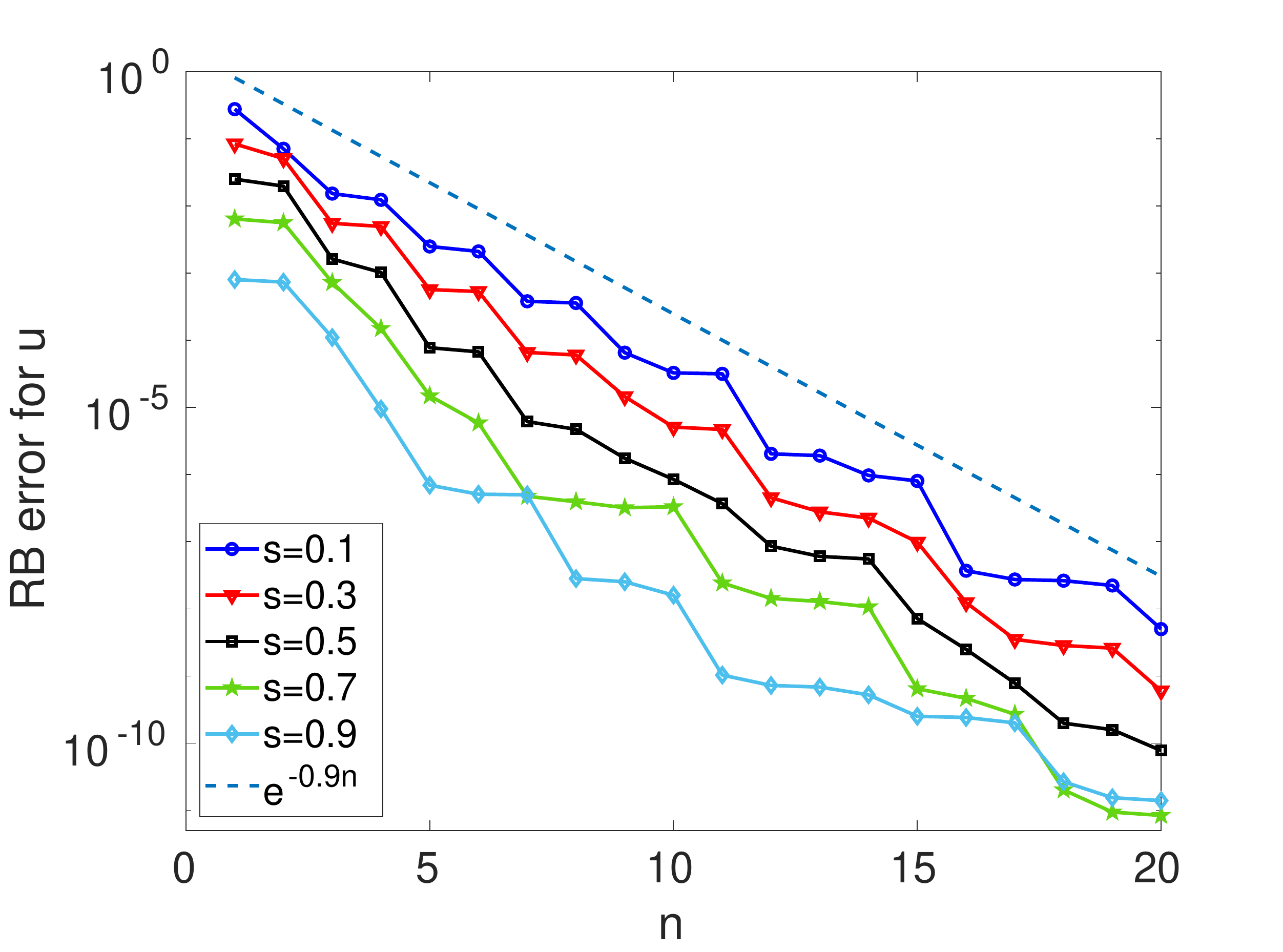}
\includegraphics[width=0.45\textwidth]{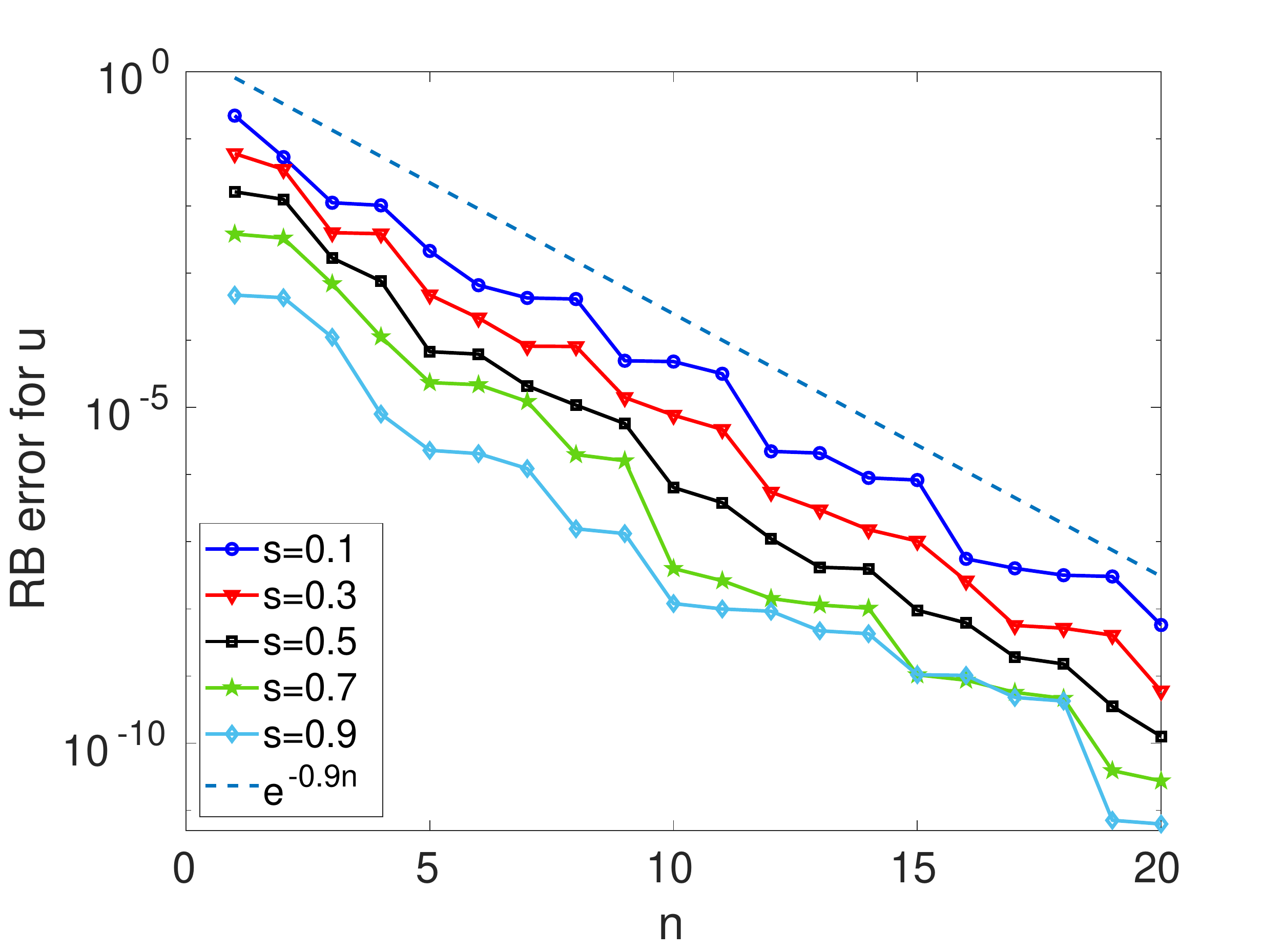}
\caption{Error $e_{u(s)}(n)$ with respect to $n$ for various values of $s$ in $[0.1,0.9]$. Left: unit square domain. Right: $L$-shape domain.} \label{fig:err_u_2D}
\end{figure}
	
\bibliographystyle{plain}
\bibliography{abrevjournal,bibliography.bib}

\end{document}